\documentclass[11pt]{article}

\usepackage[T1]{fontenc}
\usepackage{lmodern}
\usepackage{microtype}
\usepackage{amsmath,amssymb,amsthm}
\usepackage[margin=1in]{geometry}
\usepackage[colorlinks=true,linkcolor=blue,citecolor=blue,urlcolor=blue]{hyperref}

\hypersetup{
  pdftitle={Reducing Every Set of 36 Consecutive Integers to Zero by Differences of Squares},
  pdfauthor={Zhao Shen and Youran Wu},
  pdfkeywords={difference of squares, consecutive integers, reduction to zero,
    combinatorial number theory}
}

\newtheorem{theorem}{Theorem}[section]
\newtheorem{lemma}[theorem]{Lemma}
\newtheorem{corollary}[theorem]{Corollary}

\newcommand{\F}{\mathcal F}
\newcommand{\Z}{\mathbb Z}

\title{Reducing Every Set of 36 Consecutive Integers to Zero\\
by Differences of Squares}
\author{
  Zhao Shen\\
  Central South University
  \and
  Youran Wu\thanks{Email: \href{mailto:youranwu1221@outlook.com}{youranwu1221@outlook.com}.}\\
  Central South University  
}
\date{}

\begin{document}

\maketitle

\begin{abstract}

For a finite multiset of integers, repeatedly choose two entries \(a,b\)
and replace them with \(|a^{2}-b^{2}|\).  Hickerson and Kleber asked
whether, for every integer $n$, the set
\(\{n,n+1,\ldots,n+35\}\) can be reduced to zero.  We give an explicit
reduction.  Together with their results for lengths 12 and 24, this gives,
for every positive integer $L$,
\[
\{n,n+1,\dots,n+L-1\}\text{ reduces to zero for every }n\in\mathbb Z
\Longleftrightarrow 12\mid L\text{ and }L\ge 24.
\]
\end{abstract}



\section{Introduction}

For integers \(a,b\), define
\[
\mathcal{F}(a,b)=|a^{2}-b^{2}|.
\]
A reduction step removes two entries \(a,b\) from a finite multiset \(\mathcal{A}\) and inserts \(\mathcal{F}(a,b)\). Each such operation decreases the size of \(\mathcal{A}\) by one. By repeating this process until a single entry remains, if final entry can be \(0\), we say that \(\mathcal{A}\) reduces to zero.

For a positive integer $L$ and an integer $n$, write
\[
 I_L(n)=[n,n+L-1]\cap\Z
       =\{n,n+1,\ldots,n+L-1\}.
\]

Hickerson and Kleber~\cite{HK} proved that every set of 24 or 60 consecutive integers reduces to zero. Congruence conditions show that if every set of \(L\) consecutive integers reduces to zero, then \(12\mid L\). Their exhaustive search ruled out \(L=12\). Their reductions for 24 and 60 consecutive integers cover every positive multiple of 12 except 12 and 36. The case \(L=36\) was left open~\cite[Open Problem~1]{HK}. Our main result answers this remaining case.

\begin{theorem}\label{th:36}
For every $n\in\Z$, the set $I_{36}(n)$ reduces to zero.
\end{theorem}

It also gives the following classification.

\begin{corollary}\label{cor:all}
Let $L$ be a positive integer.  Then
\[
 I_L(n)\text{ reduces to zero for every }n\in\Z
 \quad\Longleftrightarrow\quad
 12\mid L\text{ and }L\ge24.
\]
\end{corollary}

The only new case needed for this classification is $L=36$.  Section~2
records the congruence obstruction and the known length-$24$ construction.
Section~3 proves Theorem~\ref{th:36} and Corollary~\ref{cor:all}.

\section{Congruences and 24 consecutive integers}

For $p\in\{2,3\}$ and $a\in\Z$, define
\[
 \mathbf 1_p(a)=
 \begin{cases}
 1,&p\nmid a,\\
 0,&p\mid a.
 \end{cases}
\]
Since every nonzero square modulo $2$ or $3$ is $1$, we have
\[
 \mathbf 1_p\bigl(\F(a,b)\bigr)
 \equiv \mathbf 1_p(a)+\mathbf 1_p(b)\pmod2.
\]

\begin{lemma}\label{lem:mod}
If $I_L(n)$ reduces to zero for every $n\in\Z$, then $12\mid L$.
\end{lemma}

\begin{proof}
For a finite multiset $\mathcal A$ and $p\in\{2,3\}$, the parity of
\[
 \sum_{a\in\mathcal A}\mathbf 1_p(a)
\]
does not change during a reduction.  It must be even if $\mathcal A$ reduces
to zero.

For $p=2$, suppose that $L=2k+1$ for an integer $k\ge0$.  The sets $I_L(0)$ and $I_L(1)$
contain $k$ and $k+1$ odd integers, respectively.  One of these numbers
is odd, a contradiction.  Thus $L$ is even.  Each $I_L(n)$ then contains
\(\frac{L}{2}\) odd integers.  Hence \(\frac{L}{2}\) is even, so
$4\mid L$.

Write $L=3q+r$, where $q$ is a nonnegative integer and
$r\in\{0,1,2\}$.  If $r=1$, choose
$n\equiv1\pmod3$.  If $r=2$, choose $n\equiv0\pmod3$.  In either case,
the set $I_L(n)$ contains exactly $2q+1$ integers not divisible by $3$.
This number is odd, which again contradicts the invariant.  Hence $r=0$ and
$3\mid L$.  Together with $4\mid L$, this gives $12\mid L$.
\end{proof}

We shall also use the following result of Hickerson and
Kleber~\cite[Lemma~4]{HK}.

\begin{lemma}\label{lem:24}
For every $n\in\Z$, the set $I_{24}(n)$ reduces to zero.
\end{lemma}

\section{Sets of 36 consecutive integers}

The construction uses one six-element identity.

\begin{lemma}\label{lem:six}
For every $x\in\Z$, the set
\[
 \{x,x+7,x+8,x+12,x+13,x+20\}
\]
reduces to zero.
\end{lemma}

\begin{proof}
The two triples give the same output:
\begin{align*}
 \F\bigl(\F(x+7,x+8),x\bigr)
   &=3|(x+5)(x+15)|,\\
 \F\bigl(\F(x+12,x+13),x+20\bigr)
   &=3|(x+5)(x+15)|.
\end{align*}
One more step gives zero.
\end{proof}

\begin{proof}[Proof of Theorem~\ref{th:36}]
Partition $\{0,1,\ldots,35\}$ into
\begin{align*}
 A_1&=\{0,1,7,11,24,28,34,35\},\\
 A_2&=\{2,9,10,14,15,22\},\\
 A_3&=\{3,4,12,16,19,23,31,32\},\\
 A_4&=\{13,20,21,25,26,33\},\\
 A_5&=\{5,6,8,17,18,27,29,30\}.
\end{align*}
These sets are disjoint and contain all thirty-six offsets.
For $1\le i\le5$, write $n+A_i=\{n+a:a\in A_i\}$.
Lemma~\ref{lem:six}, with $x=n+2$ and $x=n+13$, reduces
$n+A_2$ and $n+A_4$ to zero.

For $A_1$, the two groups of four have equal outputs:
\begin{align*}
 \F\bigl(\F(n,n+34),\F(n+1,n+35)\bigr)
   &=4624|2n+35|,\\
 \F\bigl(\F(n+7,n+24),\F(n+11,n+28)\bigr)
   &=4624|2n+35|.
\end{align*}
Thus $n+A_1$ reduces to zero.

For $A_3$, put $M=|2n+35|$ and use the pairs
\[
 (n+3,n+32),\quad (n+4,n+31),\quad
 (n+12,n+23),\quad (n+16,n+19).
\]
The four outputs are
\[
 29M,\qquad27M,\qquad11M,\qquad3M.
\]
They satisfy
\[
 \F(29M,27M)=112M^2=\F(11M,3M),
\]
so $n+A_3$ reduces to zero.

For $A_5$, two groups of four again have equal outputs:
\begin{align*}
 \F\bigl(\F(n+5,n+27),\F(n+6,n+17)\bigr)
   &=363|(2n+29)(2n+41)|,\\
 \F\bigl(\F(n+8,n+30),\F(n+18,n+29)\bigr)
   &=363|(2n+29)(2n+41)|.
\end{align*}
Thus $n+A_5$ reduces to zero.  Every block $n+A_i$ now gives zero.
Repeated use of $\F(0,0)=0$ merges these five zeros into one zero.
\end{proof}

\begin{proof}[Proof of Corollary~\ref{cor:all}]
Suppose first that every $I_L(n)$ reduces to zero.  Lemma~\ref{lem:mod}
gives $12\mid L$.  Hickerson and Kleber's exhaustive search showed that
$I_{12}(15)=\{15,16,\ldots,26\}$ does not reduce to zero
\cite[Section~4]{HK}.  Hence $L\ne12$, so $L\ge24$.

Conversely, write $L=12r$ for an integer $r\ge2$.  Choose nonnegative integers $a,b$
by
\[
 (a,b)=
 \begin{cases}
 (r/2,0),&2\mid r,\\
 ((r-3)/2,1),&2\nmid r.
 \end{cases}
\]
Then $r=2a+3b$.  Partition $I_L(n)$ into consecutive blocks: $a$ blocks
with $24$ elements and $b$ blocks with $36$ elements.  Lemma~\ref{lem:24}
and Theorem~\ref{th:36} reduce every block to zero.  Repeatedly combine the
resulting zeros.  This proves that $I_L(n)$ reduces to zero.
\end{proof}

\end{document}